\documentclass[10pt]{article}

\makeatletter
\newcommand\RedeclareMathOperator{%
  \@ifstar{\def\rmo@s{m}\rmo@redeclare}{\def\rmo@s{o}\rmo@redeclare}%
}
\newcommand\rmo@redeclare[2]{%
  \begingroup \escapechar\m@ne\xdef\@gtempa{{\string#1}}\endgroup
  \expandafter\@ifundefined\@gtempa
     {\@latex@error{\noexpand#1undefined}\@ehc}%
     \relax
  \expandafter\rmo@declmathop\rmo@s{#1}{#2}}
\newcommand\rmo@declmathop[3]{%
  \DeclareRobustCommand{#2}{\qopname\newmcodes@#1{#3}}%
}
\@onlypreamble\RedeclareMathOperator
\makeatother

\usepackage{witharrows}
\usepackage{tikz}
\usepackage{graphicx,wrapfig,lipsum}   

\usepackage{xspace}

\usepackage{framed}

\usepackage{amsmath,amssymb}
\usepackage{amsthm}
\usepackage{mathtools}
\usepackage[noend]{algorithmic}
\usepackage[ruled,vlined]{algorithm2e}
\usepackage{url}
\usepackage{fullpage}
\usepackage{makeidx}
\usepackage{enumerate}
\usepackage[top=1in, bottom=1.2in, left=0.8in, right=0.8in]{geometry}
\usepackage{graphicx,float,psfrag,epsfig,caption}

\usepackage{hyperref}
\usepackage{epstopdf}
\usepackage{color}
\usepackage{xr}
\usepackage{subcaption}
\usepackage[utf8]{inputenc}

\usepackage{scalerel,stackengine}

\usepackage{thm-restate}

\usepackage{enumitem}
\usepackage{bbm}

\stackMath
\newcommand\reallywidehat[1]{%
\savestack{\tmpbox}{\stretchto{%
  \scaleto{%
    \scalerel*[\widthof{\ensuremath{#1}}]{\kern.1pt\mathchar"0362\kern.1pt}%
    {\rule{0ex}{\textheight}}
  }{\textheight}%
}{2.4ex}}%
\stackon[-6.9pt]{#1}{\tmpbox}%
}
\usepackage[mathscr]{euscript} 

\DeclareSymbolFont{rsfs}{U}{rsfs}{m}{n}
\DeclareSymbolFontAlphabet{\mathscrsfs}{rsfs}
\usepackage{mathrsfs}

\numberwithin{equation}{section}

\newtheoremstyle{myexample} 
    {\topsep}                    
    {\topsep}                    
    {\rm }                   
    {}                           
    {\bf }                   
    {.}                          
    {.5em}                       
    {}  

\newtheoremstyle{myremark} 
    {\topsep}                    
    {\topsep}                    
    {\rm}                        
    {}                           
    {\bf}                        
    {.}                          
    {.5em}                       
    {}  

\newtheorem{claim}{Claim}[section]

\newtheorem{theorem}[claim]{Theorem}
\newtheorem{proposition}[claim]{Proposition}
\newtheorem{corollary}[claim]{Corollary}

\newtheorem{lemma}[claim]{Lemma}

\theoremstyle{myremark}
\newtheorem{remark}[claim]{Remark}

\theoremstyle{myremark}

\theoremstyle{myexample}

\definecolor{darkgreen}{rgb}{0.0, 0.5, 0.0}

\hypersetup{
  colorlinks   = true, 
  urlcolor     = blue, 
  linkcolor    = blue, 
  citecolor   = red 
}

\newcommand{\bea}{\begin{eqnarray}}
\newcommand{\eea}{\end{eqnarray}}
\newcommand{\<}{\langle}
\renewcommand{\>}{\rangle}

\newcommand{\wt}{\widetilde}

\def\eps{{\varepsilon}}

\def\bh{\boldsymbol{h}}

\def\bsigma{{\boldsymbol{\sigma}}}

\def\cG{{\mathcal G}}
\def\cT{{\mathcal T}}
\def\cC{{\mathcal C}}
\def\cX{{\mathcal X}}
\def\cZ{{\mathcal Z}}

\def\bsig{{\boldsymbol \sigma}}

\def\bv{{\boldsymbol{v}}}

\def\Par{{\sf P}}
\def\de{{\rm d}}

\def\<{\langle}
\def\>{\rangle}
\def\Tr{{\sf Tr}}

\def\Poisson{{\rm Poisson}}

\def\cM{{\cal M}}
\def\cN{{\cal N}}

\def\cY{{\cal Y}}

\def\cL{{\cal L}}

\def\P{\mathbb{P}}

\def\cD{{\cal D}}

\def\b0{{\boldsymbol{0}}}

\def\cD{{{\mathcal D}}}

\def\crit{{\rm crit}}

\renewcommand{\b}{\mathbf{b}}

\def\lt{\left}
\def\rt{\right}

\def\eps{\varepsilon}

\def\bbC{{\mathbb{C}}}
\def\bbE{{\mathbb{E}}}

\def\bbN{{\mathbb{N}}}
\def\bbP{{\mathbb{P}}}
\def\bbR{{\mathbb{R}}}
\def\bbS{{\mathbb{S}}}

\def\bbZ{{\mathbb{Z}}}

\def\cN{{\mathcal{N}}}

\def\cQ{{\mathcal{Q}}}

\def\SYK{{\mathrm{SYK}}}
\def\QSK{{\mathrm{QSK}}}
\def\End{{\mathrm{End}}}

\def\bh{{\boldsymbol{h}}}

\DeclareMathOperator*{\E}{\bbE}
\RedeclareMathOperator*{\P}{\bbP}

\newcommand{\diff}[1]{{\mathrm{d}#1}}

\title{Free Energy Subadditivity for Symmetric Random Hamiltonians}
\author{Mark Sellke\thanks{Institute for Advanced Study and Amazon Core AI. Email: \texttt{msellke@gmail.com}}}

\date{}

\begin{document}

\maketitle

\abstract{
We consider a random Hamiltonian $H:\Sigma\to\bbR$ defined on a compact space $\Sigma$ that admits a transitive action by a compact group $\cG$. When the law of $H$ is $\cG$-invariant, we show its expected free energy relative to the unique $\cG$-invariant probability measure on $\Sigma$ obeys a subadditivity property in the law of $H$ itself. The bound is often tight for weak disorder and relates free energies at different temperatures when $H$ is a Gaussian process. Many examples are discussed including branching random walks, several spin glasses, random constraint satisfaction problems, and the random field Ising model. We also provide a generalization to quantum Hamiltonians with applications to the quantum SK and SYK models.
}

\section{A General Subadditivity Result}

A large part of statistical physics and probability theory is concerned with determining the free energy of a given Hamiltonian. In many models for disorded systems, this Hamiltonian is itself a random function. This paper focuses on free energies of random Hamiltonians obeying rich distributional symmetries.

We begin with a few generalities. Let $\Sigma$ be a compact metric space equipped with reference Borel probability measure $\mu$. We consider a (random) continuous Hamiltonian function $H:\Sigma\to\bbR$ with law $\cL$. Define the associated (random) partition function $Z(H)$ and (deterministic) expected free energy $F(\cL)$ by
\begin{align}
\label{eq:Z}
    Z(H)&=\int e^{H(\bsig)} \mu(\de \bsig);
    \\
\label{eq:F}
    F(\cL)&=\bbE[\log Z(H)].
\end{align}
$F(\cL)$ will always be assumed finite.

Our main result requires that $(\Sigma,\mu,H)$ be highly symmetric in a joint sense. Precisely, we require the existence of a continuous and transitive action $\cG\times \Sigma\to\Sigma$ on $\Sigma$ by a compact group $\cG$; such an action is said to make $\Sigma$ a homogeneous space for $\cG$. In particular, each $g\in\cG$ acts on $\Sigma$ by a homeomorphism $\bsig\mapsto g\bsig$.
Since $\Sigma$ and $\cG$ are compact, there exists a unique $\cG$-invariant probability measure on $\Sigma$ under which $\bsig$ and $g\bsig$ have the same law for each $g\in\cG$ (see e.g. \cite[Theorem 6.2]{diestel2014joys}). We require $\mu$ to be this $\cG$-invariant probability measure.

We say that the law $\cL$ of $H$ is $\cG$-invariant if for each $g\in\cG$ the transformed Hamiltonian given by
\[
    H^g(\bsig)=H(g\bsig),\quad \forall\bsig\in \Sigma
\]
also has law $\cL$. Further, given two Hamiltonian distributions $\cL_1$ and $\cL_2$, let $\cL_1+\cL_2$ be the law of the independent sum $H_1+H_2$ for $(H_1,H_2)\sim \cL_1\times \cL_2$. We can now state our main subadditivity result.

\begin{theorem}
\label{thm:main}
Suppose $(\Sigma,\mu,\cG)$ are as above, and in particular that $\mu$ is the unique probability measure on $\Sigma$ invariant under the transitive action of $\cG$. Let $\cL_1,\cL_2$ be two laws for random Hamiltonians on $\Sigma$ and suppose that $\cL_1$ is $\cG$-invariant. Then
\[
    F(\cL_1+\cL_2)\leq F(\cL_1)+F(\cL_2).
\]
\end{theorem}

\begin{proof}
Consider the random probability measure $\wt\mu$ defined from $H_{1}$  by
\[
    \wt\mu(\de\bsig)
    =
    \frac{e^{H_{1}(\bsig)} \mu(\de\bsig)}{\int e^{H_{1}(\bsig)} \mu(\de \bsig)}.
\]
Note that $\bbE[\wt\mu(\de\bsig)]=\mu(\de\bsig)$, i.e. for any bounded, measurable function $f:\Sigma\to\bbR$ independent of $\wt\mu$, we have
\begin{equation}
\label{eq:measure-valued-martingale}
   \bbE\lt[\int f(\bsig)\wt\mu(\de\bsig)\rt]=\int f(\bsig)\mu(\de\bsig).
\end{equation}
Indeed, \eqref{eq:measure-valued-martingale} certainly defines \emph{some} probability measure $\mu'$ on the right-hand side. Moreover this probability measure inherits $\cG$-invariance from $\cL_1$. The uniqueness result \cite[Theorem 6.2]{diestel2014joys} mentioned above now implies $\bbE[\wt\mu(\de\bsig)]=\mu(\de\bsig)$ as claimed. Using this, we obtain
\begin{equation}
\label{eq:FN-step-1}
\begin{aligned}
    F(\cL_1+\cL_2)
    &=
    \bbE\lt[
        \log \int e^{H_1(\bsig)+H_2(\bsig)}\mu(\de \bsig)
    \rt]
    \\
    &=
    \bbE\lt[
    \log \int e^{H_1(\bsig)}\mu(\de \bsig)
    \rt]
    +
    \bbE\lt[
        \log \int \frac{e^{H_1(\bsig)+H_2(\bsig)}}
        {\int e^{H_1(\bsig)}\mu(\de \bsig)}
        \mu(\de \bsig)
    \rt]
    \\
    &=
    F(\cL_1)
    +
    \bbE\lt[
    \log
    \int e^{H_2(\bsig)}\wt\mu(\de\bsig)
    \rt]
    \\
    &
   \leq
    F(\cL_1)+F(\cL_2)
    .
\end{aligned}
\end{equation}
In the last step we used \eqref{eq:measure-valued-martingale} together with Jensen's inequality. This concludes the proof.
\end{proof}

Even without symmetry, a weaker form of the above estimate can be obtained with $F(\cL_2)$ replaced by the simple upper bound
\begin{equation}
\label{eq:simple-UB}
    \sup_{\bsig\in\Sigma}\log\mathbb E[e^{H_2(\bsig)}].
\end{equation}
Indeed by Jensen's inequality, \eqref{eq:simple-UB} upper-bounds $\bbE\lt[\log\int e^{H_2(\bsig)}\wt\mu(\de\bsig)\rt]$ for any probability measure $\wt\mu$. 
In the special case that the marginal distributions of $H_1(\bsig)$ and $H_2(\bsig)$ do not depend on $\bsig$, we similarly find that
\[
    F(\cL_1)+F(\cL_2)\leq 
    \log\mathbb E[e^{H_1(\bsig)}]+\log\mathbb E[e^{H_2(\bsig)}]
    =
    \log\mathbb E[e^{H_1(\bsig)+H_2(\bsig)}].
\]
In particular the estimate in Theorem~\ref{thm:main} is asymptotically tight if $\cL_1+\cL_2$ satisfies 
\[
    F(\cL_1+\cL_2)
    \approx
    \log \bbE[e^{H_1(\bsig)+H_2(\bsig)}].
\]
This near-equality holds at weak disorder or high temperature in many examples, including some presented in the next section. The right-hand side is often referred to as the \emph{annealed} free energy of $H_1+H_2$.

Let us also point out that specializing Theorem~\ref{thm:main} to zero temperature by replacing $(H_1,H_2)$ with $(\beta H_1,\beta H_2)$ for large $\beta$ yields only the trivial bound
\[
    \max_{\bsig\in\Sigma}\big( H_1(\bsig)+H_2(\bsig) \big)
    \leq
    \max_{\bsig\in\Sigma}H_1(\bsig)+\max_{\bsig\in\Sigma}H_2(\bsig).
\]
This bound holds with no symmetry assumption, but for finite $\beta$ the symmetry conditions in Theorem~\ref{thm:main} are essential. In fact without symmetry, simple counterexamples exist even on a two-point space $\Sigma=\{-1,1\}$ with $\mu$ uniform and $H_1=H_2$ deterministic. Taking $H_1(1)=H_2(1)=0$ and $H_1(-1)=H_2(-1)=x$ yields
\begin{align*}
    F(\cL_1)+F(\cL_2)&=2\log\lt(\frac{1+e^x}{2}\rt),
    \\
    F(\cL_1+\cL_2)&= \log\lt(\frac{1+e^{2x}}{2}\rt)
\end{align*}
and the latter is easily seen to be strictly larger for all real $x\neq 0$.

\begin{remark}
Free energy subadditivity in the \emph{system size} has been established for mean-field spin glasses and sparse graph models obeying certain convexity properties in \cite{guerra2002thermodynamic,bayati2010combinatorial,huang2018convergence,jagannath2022existence}. These results are fundamentally important as they imply the existence of a limiting asymptotic free energy or ground state value in such models. Theorem~\ref{thm:main} is in a different spirit as the state space $\Sigma$ is fixed while the Hamiltonian varies. As demonstrated by the examples in the next section, our result applies somewhat generically and does not require any convexity conditions. On the other hand, Theorem~\ref{thm:main} is a simpler bound and is usually not tight for low temperatures, while the results of \cite{guerra2002thermodynamic,bayati2010combinatorial,huang2018convergence} are asymptotically tight at all temperatures. (Indeed, our result only implies an $O(\beta^2)$ upper bound for free energies at inverse temperature $\beta$, while $\Theta(\beta)$ is the correct asymptotic in e.g. mean-field spin glass models.)
\end{remark}

\section{Examples}

We discuss several statistical physics models to which Theorem~\ref{thm:main} applies. While we focus on models that are important in their own right, many others are easily constructed by e.g. independently summing some of the random Hamiltonians presented below.

\subsection{Branching Random Walk}

Let $\cT_{N,d}$ be a $d$-ary rooted tree of depth $N$, and fix a probability distribution $\nu$ on $\bbR$ with finite exponential moments. For each vertex $v\in V(\cT_{N,d})$, generate an i.i.d. variable $x_v\sim \nu$. For a leaf $w$ of $\cT_{N,d}$, let $P(w)=(v_0,v_1,\dots,v_N=w)$ be the path to $w$ from the root $v_0$. The branching random walk Hamiltonian assigns to each $w\in V(\cT_{N,d})$ the sum 
\begin{equation}
\label{eq:branching-RW-H}
    H(w)=\sum_{v\in P(w)} x_{v}.
\end{equation}
This model has been studied in e.g. \cite{chauvin1997boltzmann,jagannath2016branching} and is also known as the directed polymer on a tree. A lot of other work including \cite{bramson1978maximal,lalley1987conditional,brunet2011branching} has resulted in a precise understanding of the extreme values, which corresponds to the zero temperature setting.

Here the appropriate transitive symmetry group $\cG$ consists of all root-preserving automorphisms of $\cT_{N,d}$. It is easy to see that the random Hamiltonian in \eqref{eq:branching-RW-H} is $\cG$-invariant for any $\nu$, and that the unique $\cG$-invariant measure on the leaf set $\Sigma=\partial\cT_{N,d}=\cT_{N,d}-\cT_{N-1,d}$ is uniform. Theorem~\ref{thm:main} implies that the associated free energy
\begin{equation}
\label{eq:tree-free-energy}
    F_N(\beta)=\bbE\log \lt(\frac{\sum_{w\in\partial\cT_{N,d}} e^{\beta H(w)}}{d^N}\rt)
\end{equation}
is a subadditive function of $\nu$. Moreover Theorem~\ref{thm:main} is asymptotically tight for small $\beta$: \cite{buffet1993directed,chauvin1997boltzmann} show that for $\beta\leq \beta_{\crit}(\mu,d)$, the limiting free energy agrees with the annealed value
\[
    F(\beta)=\lim_{N\to\infty} F_N(\beta)/N
    =
    \log \bbE^{x\sim\nu}[e^{\beta x} ]
\]
which is clearly additive in $\nu$.

Let us also point out that the generalized random energy model (GREM) Hamiltonian \cite{derrida1985generalization,derrida1986solution} takes the same form as \eqref{eq:tree-free-energy}, but with the depth $N$ fixed and the degree $d$ growing. Here the distribution of $x_v$ may depend on the depth of $v$, which doesn't affect the symmetry used above. It follows that Theorem~\ref{thm:main} applies also to the GREM free energy.

\subsection{Spin Glasses with Gaussian Disorder}
\label{subsec:mixed-p-spin}

Spin glasses give rise to some of the most canonical examples of random Hamiltonians. Theorem~\ref{thm:main} applies to the following quite general family of Ising spin glasses. Let $\Sigma_N=\{-1,1\}^N$ and for $1\leq p\leq P$ and each $(i_1,\dots,i_p)\in [N]^p$, let
\[
    J_{i_1,\dots,i_p}\sim \cN(0,c_{i_1,\dots,i_p})
\]
be a centered Gaussian with arbitrary variance $c_{i_1,\dots,i_p}\geq 0$. We assume the variables $J_{i_1,\dots,i_p}$ are jointly independent and define the Hamiltonian
\begin{equation}
\label{eq:Gaussian-H}
    H_N(\bsig)
    =
    \sum_{p=1}^P
    \frac{1}{N^{(p-1)/2}}
    \sum_{1\leq i_1,\dots,i_p\leq N}
    J_{i_1,\dots,i_p}
    \sigma_{i_1}\dots\sigma_{i_p}.
\end{equation}
To apply Theorem~\ref{thm:main} to $H_N$, we take $\cG=\bbZ_2^N$ which acts on $\Sigma$ by 
\[
    (g\bsig)_i=(-1)^{g_i}\bsig_i,\quad i\in [N].
\]
This is a transitive action preserving the law of $H_N$, and it is easy to see that the $\cG$-invariant measure $\mu$ is uniform on $\Sigma$. Moreover the independent sum $\mathcal L_1+\mathcal L_2$ amounts to adding the vectors $\vec c$ of coefficient variances.
For $\beta\geq 0$ the associated free energy at inverse temperature $\beta$ is
\[
    F_{N,\beta}(\vec c)=\bbE \log \int e^{\beta H_N(\bsig)}\mu(\de \bsig),
\]
for $\mu$ the uniform measure on $\Sigma_N$. Of course $F_{N,\beta}(\vec c)=F_{N}(\beta^2\vec c)$, so we define also $F_{N}(\vec c)=F_{N,1}(\vec c)$. Theorem~\ref{thm:main} then implies the following.
\begin{corollary}
\label{cor:temperature}
    For entrywise non-negative coefficient variances $\vec c$ and $\vec{c'}$, we have
\[
    F_N\big(\vec c+\vec{c'}\big)
    \leq
    F_N(\vec c)
    +
    F_N(\vec {c'}).
\]
In particular for $\beta_1,\beta_2\geq 0$ we have the subadditivity in $\beta^2$ bound
\begin{equation}
\label{eq:beta}
    F_{N,\sqrt{\beta_1^2+\beta_2^2}}\lt(\vec c\rt)
    \leq
    F_{N,\beta_1}(\vec c)
    +
    F_{N,\beta_2}(\vec c).
\end{equation}
\end{corollary}

The above setup includes many spin glasses, all of which thus satisfy Corollary~\ref{cor:temperature}. When $c_{i_1,\dots,i_p}=c_p$ depends only on $p$, we recover the well studied mixed $p$-spin model (see e.g. \cite{panchenko2013sherrington}) whose covariance 
\[
    \xi(\bsig^1,\bsig^2)
    =
    \bbE[H(\bsig^1) H(\bsig^2)]
    =
    \sum_{p=1}^P c_p \lt(\frac{\langle \bsig^1,\bsig^2\rangle}{N}\rt)^p
\]
depends only on the overlap $\langle \bsig^1,\bsig^2\rangle/N$. In this case we may write $F_{N,\beta}(\vec c)=F_{N,\beta}(\xi)$. Hence we find that for Ising mixed $p$-spin models,
\begin{align}
\label{eq:pspin-subadditive}
    F_N(\xi_1+\xi_2)&\leq F_N(\xi_1)+F_N(\xi_2)
    ,
    \\
\label{eq:pspin-temperature} 
    F_{N,\sqrt{\beta_1^2+\beta_2^2}}\lt(\xi\rt)
    &\leq
    F_{N,\beta_1}(\xi)
    +
    F_{N,\beta_2}(\xi).
\end{align}
When $c_1=0$ and $\beta\leq \beta_{\crit}(c_2,\dots,c_P)$ is small, the replica-symmetric formula 
\[
    \lim_{N\to\infty} \frac{F_N(\beta \xi)}{N}=\beta^2 \xi(1)
\]
holds. Thus \eqref{eq:pspin-temperature}  is asymptotically sharp as $N\to\infty$ for $\beta_1^2+\beta_2^2\leq \beta_{\crit}^2$ in this case. Another notable example of the form \eqref{eq:Gaussian-H} is the Edwards-Anderson model \cite{edwards1975theory,newman1997metastate} in which $(c_{v_1,v_2})_{v_1,v_2\in V(G)}$ is the adjacency matrix of a finite lattice. 
We mention that \cite{contucci2007correlation} considers the family of Hamiltonians \eqref{eq:Gaussian-H} and shows that $F(\vec c)$ is coordinate-wise increasing in $\vec c$.

We remark that for mixed $p$-spin models, we were surprisingly unable to prove \eqref{eq:pspin-subadditive} in the $N\to\infty$ limit directly from the (enormously more difficult) Parisi formula. In the appendix we derive from the Auffinger-Chen representation of the Parisi formula the weaker statement that free energy relative to counting measure on $\{-1,1\}^N$ is subadditive in $\xi$.

Instead of Ising spins $\Sigma_N=\{-1,1\}^N$, one can also consider a spherical state space or even a product of spheres. Since the Hamiltonian must now be invariant relative to a larger symmetry group, the applications are not quite as general as above. A still broad family of examples comes from setting
\[
    \Sigma=\prod_{i=1}^r \mathbb S^{N_i-1}(\sqrt{N_i})
\]
for $N=\sum_{i=1}^r N_i$, where $\mathbb S^{N_i-1}(\sqrt{N_i})$ is the sphere in $\mathbb R^{N_i}$ of radius $\sqrt{N_i}$. The corresponding transitive symmetry group $\cG= \prod_{i=1}^r O(N_i)$ is a product of orthogonal groups $O(N_i)$ with the natural action, and the unique invariant probability measure $\mu$ is the natural choice of uniform measure.
Let $s:[N]\to [r]$ send the $N_i$ coordinates in the $\mathbb S^{N_i-1}(\sqrt{N_i})$ factor above to $i\in [r]$.
It is not difficult to see that Hamiltonians in \eqref{eq:Gaussian-H} are $\cG$-invariant if $c_{i_1,\dots,i_p}=c_{s(i_1),\dots,s(i_p)}$. Therefore the subadditivity results in Corollary~\ref{cor:temperature} hold for such $\vec c$. These models are known as \emph{multi-species} spin glasses and are the subject of much recent work \cite{barra2015multi,panchenko2015free,mourrat2021nonconvex,kivimae2021ground,subag2021tap1,subag2021tap,subag2021second,bates2022free,bates2021crisanti}.

Yet another generalization is to consider pairs $(\bsig^1,\bsig^2)\in\bbS^{N-1}(\sqrt N)$ with the overlap constraint $\langle \bsig^1,\bsig^2\rangle = RN$ for some $R\in [-1,1]$. Our discussion above goes through essentially unchanged for $\bsig^1,\bsig^2$ on the sphere $\bbS^{N-1}(\sqrt N)$ because all such pairs are related by a simultaneous $O(N)$ action. For Ising spin glasses, the same is true once $\bbZ_2^N$ is replaced by the full symmetry group of the cube which includes also coordinate permutations. In both these cases, the analog of Corollary~\ref{cor:temperature} applies to the two-replica Hamiltonian  
\begin{equation}
\label{eq:two-replica}
    \overline{H}_N(\bsig^1,\bsig^2)
    =
    H_N(\bsig^1)+H_N(\bsig^2).
\end{equation}
Thus we deduce subadditivity in $\vec c$ for 
\[
    F_{N,\beta}(\vec c)\equiv 
    \bbE
    \log
    \int e^{\beta H_N(\bsig^1)+ \beta H_N(\bsig^2)}
    \mu_R(d(\bsig^1,\bsig^2)),
\]
where $\mu_R$ is the unique $\cG$-invariant measure on pairs $(\bsig^1,\bsig^2)\in\bbS^{N-1}(\sqrt N)\times \bbS^{N-1}(\sqrt N)$ with $\langle \bsig^1,\bsig^2\rangle = RN$. 

The free energies of such constrained pairs appear in Talagrand's proof \cite{talagrand2006parisi} of the Parisi formula and have also been used to study disorder chaos \cite{chen2014chaos,chen2017parisi,chen2015disorder,chen2017variational,chen2018disorder,chen2018energy}. They are also a special case of the so-called vector spin models \cite{panchenko2007overlap,panchenko2018free,ko2020free}. Further generalizations with multiple correlated Hamiltonians and more replicas have been used in relation with the overlap gap property to establish computational barriers against efficiently optimizing $H_N$ \cite{chen2019suboptimality,GJW20,huang2021tight} -- on the sphere, these are still symmetric enough to conclude subadditivity, but we omit the details.

\subsection{Orthogonally Invariant Spin Glasses}

The orthogonally invariant Sherrington-Kirkpatrick model is a spin glass with dependent couplings studied in \cite{marinari1994replica,parisi1995mean,bhattacharya2016high,fan2021replica,barbier2022marginals,fan2022tap}. To define it, one chooses a diagonal $N\times N$ matrix $\Lambda_N$ and sets
\[
    A_N=O_N \Lambda_N O_N^{\top}
\]
for a Haar-random orthogonal matrix $O_N\in O(N)$. The associated random Hamiltonian is $H_N(\bsig)=\langle \bsig, A_N\bsig\rangle$ for $\bsig\in\Sigma_N=\{-1,1\}^N$. We let $F_N(\nu)$ be the corresponding free energy when the entries of $\Lambda$ are drawn i.i.d. from the compactly supported measure $\nu$. Such Hamiltonians are $O(N)$-invariant by definition, and in particular under the natural $\bbZ_2^N$-action on $\Sigma_N$. Hence the conditions of Theorem~\ref{thm:main} are met.

In this case, addition of Hamiltonians corresponds to additive free convolution. Indeed if the entries of $\Lambda_{N,i}$ are drawn i.i.d. from $\nu_i$ for $i\in \{1,2\}$, and if $O_{N,1},O_{N,2}\in O(N)$ are Haar-random and independent, then the sum 
\[
    A_{N,1}+A_{N,2}=O_{N,1}\Lambda_{N,1} O_{N,1}^{\top}+O_{N,2}\Lambda_{N,2} O_{N,2}^{\top}
\]
is orthogonally invariant with random spectrum converging in probability (in e.g. $W_2$ metric) to the additive free convolution $\nu_1\boxplus \nu_2$ as $N\to\infty$. Moreover it follows from \cite[Proof of Proposition 1.1]{bhattacharya2016high} that the normalized free energy of the orthogonally invariant SK model is Wasserstein-continuous in the spectrum, leading to the approximate subadditivity relation
\begin{equation}
\label{eq:free-conv-ineq}
    F_N(\nu_1\boxplus\nu_2)\leq F_N(\nu_1)+F_N(\nu_2)+o_N(N)
\end{equation}
for the orthogonally invariant SK model. As in the previous subsection, this bound is asymptotically tight in the replica-symmetric phase. Indeed here the limiting free energy $F(\nu)=\lim_{N\to\infty} \frac{F_N(\nu)}{N}$ agrees with the annealed value, which is given by an integrated $R$-transform of $\nu$ (\cite[Theorem 1.2 and Equation (1.8)]{bhattacharya2016high}) and is therefore additive under free convolution.

A similar construction is possible for tensors as well. For $p\geq 3$ one may construct a random $p$-tensor $A_N^{(p)}$ by starting with a deterministic $p$-tensor $\Lambda_N^{(p)}$ and conjugating on all $p$ ``sides" by an independent Haar-random matrix $O_N^{(p)}\in O(N)$. The resulting orthogonally invariant mixed $p$-spin model Hamiltonian
\[
    H_N(\bsig)=\sum_{p=2}^P \langle A_N^{(p)},\bsig^{\otimes p}\rangle
\]
has law which is $O(N)$-invariant and hence $\bbZ_2^N$-invariant.
Thus for deterministic sequences $(\Lambda_N^{(2)},\dots,\Lambda_N^{(P)})$, the subadditivity result Theorem~\ref{thm:main} still applies. However it is unclear how to extend the clean statement \eqref{eq:free-conv-ineq} beyond matrices; this would probably require a theory of free probability for tensors.

\subsection{Random Constraint Satisfaction Problems}

In a typical random constraint satisfaction problem, one equips the state space $\Sigma_{N}=\{-1,1\}^N$ with Hamiltonian
\[
    H_N(\bsig)=-\sum_{j=1}^M\theta_j(\bsig_{i_1^j},\dots,\bsig_{i_k^j})
\]
for $M=\alpha N$ or $M\sim \Poisson(\alpha N)$. Here each $\theta_j:\{-1,1\}^k\to \bbR_{\geq 0}$ is an i.i.d. non-negative function invariant in law under the action of $\bbZ_2^N$, and the indices $i_{\ell}^j\in[N]$ are i.i.d. as well. For example random $k$-SAT and NAE $k$-SAT can be represented in the above way, where $\theta_j$ equals $0$ when the corresponding clause is satisfied and $1$ otherwise. These models have been studied in great detail \cite{achlioptas2004threshold,coja2016asymptotic,ding2016satisfiability,ding2022proof,nam2022one}.

In this setting, applying Theorem~\ref{thm:main} shows the free energy is subadditive in the clause density $\alpha$. Specialized to $\beta H_N$ for large $\beta$, this roughly says that the typical solution density is submultiplicative in $\alpha$, which is easy to see directly. (However technically we cannot set $\beta=\infty$ in Theorem~\ref{thm:main} to obtain ``hard" constraints as then $\log(Z)=\log(0)$ holds with positive probability.)

The random Ising perceptron model (see \cite{talagrand2011mean1,talagrand2011mean2,ding2019capacity}) is similarly $\bbZ_2^N$-invariant. Here $k=N$ with deterministic indices $(i_1^j,\dots,i_N^j)=(1,2,\dots,N)$ and the functions take the form $\theta_j(\bsig)=\varphi(\langle g_j,\bsig\rangle)$ for a deterministic function $\varphi:\bbR\to\bbR$ and Gaussian disorder vector $g_j\sim \cN(0,I_N)$. Theorem~\ref{thm:main} applies also to this model, as well as to the spherical analog where $\Sigma= \bbS^{N-1}(\sqrt N)$ is equipped with uniform measure.

\subsection{Random Field Ising Model and Spiked Matrices}

The random field Ising model (RFIM) was introduced by \cite{imry1975random}, see also \cite{bricmont1988phase,aizenman1990rounding,chatterjee2019central,ding2021exponential}. Here one considers a vertex set such as $V=[N]^d$ with ferromagnetic nearest neighbor interactions and a random external field $\bh\in \bbR^V$, giving rise to a Hamiltonian with independent two contributions:
\begin{align}
\label{eq:RFIM}
    H(\bsig)&=
    H_1(\bsig)+H_2(\bsig)
    ,
    \quad
    \bsig\in \{-1,1\}^V;
    \\
\nonumber
    H_1(\bsig)&=\sum_{v\in V}
    h_v \sigma_v,
    \\
\nonumber
    H_2(\bsig)&=\sum_{(v,w)\in V:~\|v-w\|=1}
    \sigma_v \sigma_w.
\end{align}

The random field term $\sum_{v\in V} h_v \sigma_v$ is invariant under the natural $\bbZ_2^V$ action of $\{-1,1\}^V$ if the values $(h_v)_{v\in V}$ are jointly independent and each $h_v$ has negation-invariant law. Since Theorem~\ref{thm:main} requires only that $H_1$ be $\cG$-invariant, this implies an upper bound on the free energy in \eqref{eq:RFIM} with $\mu$ the uniform measure on $\{-1,1\}^V$. As the expected free energy of $H_1$ is just $\sum_{v\in V} \bbE[\log \cosh(h_v)]$, we obtain an elementary upper bound on the RFIM free energy in terms of the ordinary Ising model free energy. The same techinque also gives a simpler proof of the lower bound in \cite[Theorem A.1]{frohlich1987some}.

Another natural example of a random Hamiltonian which is a sum of two different terms is the spiked matrix or tensor model which has been studied extensively in e.g. \cite{johnstone2001distribution,baik2005phase,benaych2011eigenvalues,donoho2018optimal,perry2018optimality,perry2016statistical,chen2019phase,arous2019landscape}. The spiked matrix model is defined by starting with a random matrix $A$ and adding a random rank $1$ spike to obtain $A+\bv\bv^{\top}$. The resulting Hamiltonian is $H(\bsig)=\langle \bsig,A\bsig\rangle+\langle \bsigma,\bv\rangle^2$. If for example $\bv$ is uniform on the sphere or $\{-1,1\}^N$, then the spike is $O(N)$ or $\bbZ_2^N$ invariant. As explained in e.g. \cite{chen2019phase}, the free energy in a spiked model is intimately related with detectability of the spike.
In these models as well, Theorem~\ref{thm:main} applies when either $A$ or $\bv$ obeys the requisite symmetry property and gives a simple free energy upper bound separating the different interactions in the Hamiltonian.


\section{Generalization to Quantum Hamiltonians}

A version of our result extends to quantum Hamiltonians. We restrict attention to operators on finite dimensional spaces, though this is probably not essential. Let $M$ be a random $N\times N$ Hermitian matrix with law $\cM$. Then the associated quantum partition function and average free energy are given by
\begin{align}
\label{eq:Z-M}
    Z(M)&=\Tr\lt(e^M\rt)/N;
    \\
\label{eq:F-M}
    F(\cM)&=\bbE^{M\sim \cM}[\log Z(M)].
\end{align}
In the case that $M$ is almost surely diagonal in a fixed orthogonal basis $(v_1,\dots,v_N)$, $M$ can be viewed as a classical Hamiltonian taking value $\lambda_i(M)$ on the state $\bsig=v_i$. $Z(M)$ and $F(\cM)$ defined above then agree with the classical free energy for the uniform reference measure $\mu$ on $\{v_1,\dots,v_N\}$.

Similarly to before, let $\cM_1+\cM_2$ denote the law of the independent matrix sum $M_1+M_2$ for $(M_1,M_2)\sim \cM_1\times \cM_2$. 
We say that $\cM$ is $\cG$-invariant for a compact group $\cG\subseteq U(N)$ of unitary transformations if for $M\sim \cM$ independent of any fixed $g\in\cG$, we have $gM\sim \cM$.

Our generalization of the symmetry assumption is as follows. With $\bbE^{g\in \cG}$ denoting expectation relative to Haar measure, we say that $\cG$ is \emph{$M$-symmetrizing} if
\begin{equation}
\label{eq:invariant-M}
    \bbE^{g\in\cG}[gMg^{-1}]= \frac{\Tr(M) I_N}{N}.
\end{equation}
(Here $\Tr(M)$ is a scalar and $I_N$ denotes the identity matrix.)
For diagonal $M$, the group $S_N\subseteq U(N)$ of permutation matrices is $M$-symmetrizing, as are all of its transitive subgroups. This recovers the setting of Theorem~\ref{thm:main} for finite $\Sigma$. As another example, the subgroup of \textbf{signed} permutation matrices is $M$-symmetrizing for all not-necessary-diagonal $M$. With these definitions in place, we can now state a generalization of Theorem~\ref{thm:main} to quantum free energies.

\begin{theorem}
\label{thm:main-M}
Suppose $\cM_1$ is $\cG$-invariant and $\cG$ is almost surely $e^{M_1}$-symmetrizing for $M_1\sim\cM_1$. Then
\[
    F(\cM_1+\cM_2)\leq F(\cM_1)+F(\cM_2).
\]
\end{theorem}

\begin{proof}
    Recall the Golden-Thompson inequality \cite{golden1965lower,thompson1965inequality,forrester2014golden}: for all Hermitian matrices $M_1$ and $M_2$,
    \begin{equation}
    \label{eq:GT}
        \Tr(e^{M_1+M_2})\leq \Tr(e^{M_1} e^{M_2})\,.
    \end{equation}
    Using Jensen's inequality in the final step below, we find:
\begin{align*}
    F(\cM_1+\cM_2)
    &=
    \bbE\lt[
        \log \Tr\lt(e^{M_1+M_2}\rt)/N
    \rt]
    \\
    &=
    \bbE\lt[
        \log \Tr\lt(e^{M_1}\rt)/N
    \rt]
    +
    \bbE\lt[
        \log \frac{\Tr\lt(e^{M_1+M_2}\rt)}{ \Tr\lt(e^{M_1}\rt)}
    \rt]
    \\
    &\stackrel{\text{\eqref{eq:GT}}}{\leq}
    F(\cM_1)
    +
    \bbE\lt[
        \log \frac{\Tr\lt(e^{M_1}e^{M_2}\rt)}{\Tr\lt(e^{M_1}\rt)}
    \rt]
    \\
    &\leq
    F(\cM_1)
    +
    \bbE^{M_2\sim\cM_2}\lt[
        \log  \bbE^{M_1\sim\cM_1}\lt[\frac{\Tr\lt(e^{M_1}e^{M_2}\rt)}{\Tr\lt(e^{M_1}\rt)}
        \rt]
    \rt]
    .
\end{align*}
Next recall that $e^{gM_1g^{-1}}=g e^{M_1}g^{-1}$ for unitary $g$ and in particular $\Tr(e^{gM_1g^{-1}})=\Tr(e^{M_1})$. Using $\cG$-invariance of $\cM_1$ in the first step, we find that for any fixed $M_2$,
\begin{align*}
    \bbE^{M_1\sim\cM_1}
    \lt[
    \frac{\Tr\lt(e^{M_1}e^{M_2}\rt)}{\Tr\lt(e^{M_1}\rt)}
    \rt]
    &=
    \bbE^{M_1\sim\cM_1}
    \bbE^{g\in\cG}
    \lt[
    \frac{\Tr\lt(e^{gM_1g^{-1}}e^{M_2}\rt)}{\Tr\lt(e^{gM_1g^{-1}}\rt)}
    \rt]
    =
    \bbE^{M_1\sim\cM_1}
    \bbE^{g\in\cG}
    \lt[
    \frac{\Tr\lt(e^{gM_1g^{-1}}e^{M_2}\rt)}{\Tr\lt(e^{M_1}\rt)}
    \rt]
    \\
    &=
    \bbE^{M_1\sim\cM_1}
    \lt[
    \frac{\Tr\lt( \bbE^{g\in\cG}[e^{gM_1g^{-1}}]e^{M_2}\rt)}{\Tr\lt(e^{M_1}\rt)}
    \rt]
    =
    \bbE^{M_1\sim\cM_1}
    \lt[
    \frac{\Tr\lt( \bbE^{g\in\cG}[ge^{M_1}g^{-1}]e^{M_2}\rt)}{\Tr\lt(e^{M_1}\rt)}
    \rt]
    \\
    &\stackrel{\eqref{eq:invariant-M}}{=}
    \bbE^{M_1\sim\cM_1}
    \lt[
    \frac{\Tr(e^{M_1})\cdot \Tr\lt(
    e^{M_2}\rt)}{N\cdot \Tr\lt(e^{M_1}\rt)}
    \rt]
    =
    \Tr(e^{M_2})/N.
\end{align*}
Combining the above displays completes the proof.
\end{proof}

In the next two subsections, we explain how to apply Theorem~\ref{thm:main-M} to the quantum SK and SYK models. In Remark~\ref{rem:quantum-general}, we briefly explain a more complicated model that subsumes both while still obeying subadditivity.

\subsection{Application to the Quantum SK Model}

The quantum SK model was introduced in \cite{ray1989sherrington,goldschmidt1990ising}. Following rigorous results in \cite{crawford2007thermodynamics,leschke2021free}, its free energy was determined at all temperatures for constant transverse field through a connection to (classical) vector spin glasses in \cite{adhikari2020free}. We consider a variant with Gaussian transverse field. To define the model, we first recall the $2\times 2$ Pauli matrices
\begin{equation}
\label{eq:pauli}
    \sigma^x
    =
    \begin{pmatrix}
    0&1\\
    1&0
    \end{pmatrix}
    ,\qquad
    \sigma^y
    =
    \begin{pmatrix}
    0&-i\\
    i&0
    \end{pmatrix}
    ,\qquad
    \sigma^z
    =
    \begin{pmatrix}
    1&0\\
    0&-1
    \end{pmatrix}
    \,.
\end{equation}
Together with the identity $I_2$, the Pauli matrices form a basis for the $2\times 2$ Hermitian matrices. Moreover they each square to $I_2$ and satisfy the relations
\begin{equation}
\label{eq:pauli-relations}
\begin{aligned}
    \sigma^x \sigma^y 
    &=
    i \sigma^z
    =
    -\sigma^y \sigma^x
    \\
    \sigma^y \sigma^z
    &=
    i \sigma^x
    =
    -\sigma^z \sigma^y
    \\
    \sigma^z \sigma^x
    &=
    i \sigma^y
    =
    -\sigma^x \sigma^z.
\end{aligned}
\end{equation}

Fix a positive integer $m$, the number of interacting particles. For $i\in [m]$, let $\sigma^x_i:(\bbC^2)^{\otimes m}\to (\bbC^2)^{\otimes m}$ be the linear operator which acts by $\sigma^x$ on the $i$-th tensor factor, and by the identity on the others:
\[
    \sigma^x_i 
    =
    I_2\otimes\dots\otimes \sigma^x\otimes\dots\otimes I_2.  
\]
Similarly define $\sigma^y_i$ and $\sigma^z_i$. The quantum SK Hamiltonian with Gaussian transverse field is the random operator $M_m:(\bbC^2)^{\otimes m}\to (\bbC^2)^{\otimes m}$ given by
\begin{equation}
\label{eq:QSK-def}
    M_m
    =
    \frac{\beta}{\sqrt m}
    \sum_{1\leq i,j\leq m}
    J_{i,j} \sigma^z_i \sigma^z_j
    +
    h\sum_{i=1}^m J_i\sigma^x_i.
\end{equation}
Here $\beta$ and $h$ are constants while as usual $J_{i,j}$ and $J_i$ are i.i.d. standard Gaussians. The corresponding free energy is 
\[
    F_m^{\QSK}(\beta,h)
    =
    \bbE
    \log\lt(
    \Tr(e^{M_m})/2^m
    \rt).
\]
It should be noted that taking $h=0$ recovers the classical SK model, since then $M_m$ is diagonal in the standard basis. However the transverse field $h\sum_{i=1}^m J_i\sigma^x_i$ behaves differently from a classical external field.

For $i\in [m]$, define the finite group
\[
    \cQ_i=\{\pm I_2,\pm \sigma^x_i,\pm\sigma^y_i,\pm\sigma^z_i\}
\]
of matrices acting on the $i$-th tensor factor. Let $\cQ\subseteq \End\big((\bbC^2)^{\otimes m}\big)$ be the group generated by all of the $\cQ_i$, so that $|\cQ|=2^{2m+1}$ (there are four choices for the operator in each tensor factor, as well as a global choice of sign). Applying Theorem~\ref{thm:main-M} to the quantum SK Hamiltonian with symmetry group $\cQ$ yields the following subadditivity result for the quantum SK model.

\begin{corollary}
\label{cor:QSK}
    For any $m\geq 1$ and constants $\beta_1,\beta_2,h_1,h_2\geq 0$, we have
    \[
        F_m^{QSK}\lt(\sqrt{\beta_1^2+\beta_2^2},\sqrt{h_1^2+h_2^2}\rt)
        \leq
        F_m^{QSK}(\beta_1,h_1)
        +
        F_m^{QSK}(\beta_2,h_2).
    \]
\end{corollary}

\begin{proof}
It is easy to see that independently summing quantum SK Hamiltonians with parameters $(\beta_1,h_1)$ and $(\beta_2,h_2)$ gives another with parameters $\lt(\sqrt{\beta_1^2+\beta_2^2},\sqrt{h_1^2+h_2^2}\rt)$, so it suffices to check that the quantum SK model verifies the conditions of Theorem~\ref{thm:main-M} with symmetry group $\cQ$.

First, note that for any $g\in\cQ$, if $M_m$ is as in \eqref{eq:QSK-def}, then the relations \eqref{eq:pauli-relations} imply
\[
    gM_m g^{-1}
    =
    \frac{\beta}{\sqrt m}
    \sum_{1\leq i,j\leq m}
    \eps_{i,j}(g)J_{i,j} \sigma^z_i \sigma^z_j
    +
    h\sum_{i=1}^m  \eps_{i} J_i (g)\sigma^x_i.
\]
for deterministic $\eps_{i,j}(g),\,\eps_i(g)\in \{-1,1\}$. Since $J_{i,j}$ and $J_i$ are negation-invariant in law and independent, we find that $gM_m g^{-1}$ has the same law as $M_m$, verifying the $\cQ$-invariance condition of Theorem~\ref{thm:main-M}.

Next we argue that all of $(\bbC^2)^{\otimes m}$ is symmetrized by $\cQ$ in the sense of \eqref{eq:invariant-M}. It is not difficult to see that a basis for $\End\lt((\bbC^2)^{\otimes m}\rt)$ consists of the operators
\begin{equation}
\label{eq:M-otimes}
    M=\sigma_1\otimes\sigma_2\otimes\dots\otimes \sigma_m
\end{equation}
where $\sigma_i \in \{I_2,\sigma^x,\sigma^y,\sigma^z\}$ for each $i\in [m]$. Moreover all of these are traceless except the identity $I_2^{\otimes m}=I_{2^m}$ which clearly satisfies \eqref{eq:invariant-M}. It remains to show that if $\sigma_i \in \{\sigma^x,\sigma^y,\sigma^z\}$ for some $i\in [n]$, then $M$ in \eqref{eq:M-otimes} satisfies
\begin{equation}
\label{eq:QSK-symmetrized}
     \bbE^{g\in\cQ}[gMg^{-1}]=0.
\end{equation}
This in turn follows by considering the effect of conjugation on the $i$-th tensor factor. Indeed since each of $\sigma^x_i,\sigma^y_i,\sigma^z_i$ commutes with half of $\cQ_i$ and anti-commutes with the other half, we find that
\begin{equation}
\label{eq:gQg}
    \sum_{g\in\cQ_i}g \sigma g^{-1}=0
\end{equation}
for $\sigma\in \{\sigma^x_i,\sigma^y_i,\sigma^z_i\}$ which implies \eqref{eq:QSK-symmetrized}.
\end{proof}

\subsection{Application to the SYK Model}

The Sachdev-Ye-Kitaev (SYK) model introduced in \cite{sachdev1993gapless,SYK-talk} is a quantum mechanical model for strongly interacting fermions with applications to holography and black holes \cite{rosenhaus2019introduction}. While different from the quantum SK model, it also takes a similar form to classical spin glass Hamiltonians such as \eqref{eq:Gaussian-H}, again with non-commuting variables. We will see that the SYK model also falls under the purview of Theorem~\ref{thm:main-M}.

To define the SYK model, we work in the involutive $\bbC$-algebra $\cC_n$ generated by $\chi_1,\dots,\chi_n$ with relations:
\begin{equation}
\label{eq:SYK-relations}
\begin{aligned}
    \chi_i&=\chi_i^*,\quad i\in [n]
    \\
    \chi_i^2 &= 1,\quad i\in [n]
    \\
    \chi_i\chi_j+\chi_j\chi_i& = 0,\quad 1\leq i<j\leq n.
\end{aligned}
\end{equation}
Its dimension is $\dim(\cC_n)=2^n$, and a basis is given by the set of products
\begin{equation}
\label{eq:chiS}
    \chi_S=\chi_{s_1}\dots\chi_{s_k}
\end{equation}
for subsets $S=\{s_1,\dots,s_k\}$ with $1\leq s_1<\dots<s_k\leq n$. 

Being an associative algebra, $\cC_n$ acts on itself by left-multiplication. This representation yields a canonical identification of each element $x\in\cC_n$ with a linear operator $M^x:\cC_n\to\cC_n$. 
Fixing a positive even integer $q\in 2\bbN$, the SYK Hamiltonian is given by such an operator $M^x$ for random $x$. Precisely, we set
\begin{equation}
\label{eq:SYK-original}
    x=
    i^{q/2}\sum_{1\leq i_1<\dots<i_q\leq n} J_{i_1,\dots,i_q} \chi_1 \dots\chi_q
\end{equation}
where the disorder variables $J_{i_1,i_2,\dots,i_q}$ are i.i.d. centered Gaussians. It is easy to verify using the relations \eqref{eq:SYK-relations} that $M^x$ is Hermitian.

Recalling \eqref{eq:F-M}, let 
\[
    F_{n,q}^{\SYK}(\beta)=\bbE \log \lt(\Tr(e^{\beta M^x})/2^n\rt)
\]
be the expected free energy of $M^x$ in \eqref{eq:SYK-original}. Our main result for the SYK model follows.
\begin{corollary}
\label{cor:SYK}
For all $n,q\geq 1$ with $q$ even and any $\beta_1,\beta_2\geq 0$,
\[
    F_{n,q}^{\SYK}\big(\sqrt{\beta_1^2+\beta_2^2}\big)
    \leq
    F_{n,q}^{\SYK}(\beta_1)
    +
    F_{n,q}^{\SYK}(\beta_2).
\]
\end{corollary}

The proof of Corollary~\ref{cor:SYK} is of course based on Theorem~\ref{thm:main-M}. 
Let $\cG_{\SYK}^+=\{\chi_S\}_{S\subseteq [n]}$ consist of all $2^n$ monomials written in sorted order, which is easily seen to be a basis for $\cC_n$. $\cG_{\SYK}^+$ is not a group, but $\cG_{\SYK}=\cG_{\SYK}^+\cup (-\cG_{\SYK}^+)$ is a group as shown in Proposition~\ref{prop:quantum-group} below. We will apply Theorem~\ref{thm:main-M} with $\cG=\cG_{\SYK}$. Corollary~\ref{cor:SYK} thus reduces to the following two claims.

\begin{lemma}
\label{lem:M-invariant}
The law of $M^x$ is $\cG_{\SYK}$-invariant.
\end{lemma}

\begin{lemma}
\label{lem:M-symmetrizing}
$\cG_{\SYK}$ is $M^x$-symmetrizing for all $x\in\cC_n$.
\end{lemma}

We begin with an easy preliminary result. For any permutation $\pi\in S_n$ and $S=\{s_1,\dots,s_k\}\subseteq [n]$ with $1\leq s_1<s_2<\dots<s_k\leq n$, define
\[
    \chi^{\pi}_S
    =
    \chi_{\pi(s_1)}\chi_{\pi(s_2)}\dots\chi_{\pi(s_k)}.
\]

\begin{proposition}
\label{prop:quantum-group}
    For any $\pi\in S_n$, the set of $2^{n+1}$ monomials $\{\pm\chi^{\pi}_S\}_{S\subseteq [n]}$ coincides with the multiplicative subgroup $\cG_{\SYK}\subseteq \cC_n$ and is in particular independent of $\pi$.
\end{proposition}

\begin{proof}
    First we show that $\cG_{\SYK}$ is a group. It is easy to see by using the relations \eqref{eq:SYK-relations} to reorder terms that any product of (possibly non-distinct) generators $\chi_{i_1}\dots\chi_{i_k}$ takes the form $\pm\chi_S$ for some $S\subseteq [n]$, so that $\cG_{\SYK}$ is closed under multiplication. Similar reasoning reveals that $\chi_S^{-1}\in \{\chi_S,-\chi_S\}$ for all $S\in [n]$, so that $\cG_{\SYK}$ is closed under inverse and is hence a group.
    
    Next, let $\cG_{\SYK}^{\pi}\subseteq \cC_n$ denote the set $\{\pm\chi^{\pi}_S\}_{S\subseteq [n]}$. We claim that $\cG_{\SYK}^{\pi}=\cG_{\SYK}^{\pi'}$ if $\pi'$ is obtained from $\pi$ by an adjacent transposition, which implies that $\cG_{\SYK}^{\pi}$ does not depend on the permutation $\pi$. This in turn follows from the equality of sets $\{\chi_i\chi_j,-\chi_i\chi_j\}=\{\chi_j\chi_i,-\chi_j\chi_i\}$ for distinct $i,j\in [n]$, completing the proof.
\end{proof}

\begin{proof}[Proof of Lemma~\ref{lem:M-invariant}]
Fix any monomial $y=\pm\chi_{j_1}\dots\chi_{j_m}$ and let
\[
    x=
    i^{q/2}
    \sum_{1\leq i_1<\dots<i_q\leq n} J_{i_1,\dots,i_q} \chi_1 \dots\chi_q
\]
be as in \eqref{eq:SYK-original}. Note that $y^{-1}\in \{y,-y\}$. It follows from the relations~\eqref{eq:SYK-relations} that 
\[
    yxy^{-1}
    =
    i^{q/2}\sum_{1\leq i_1<\dots<i_q\leq n} \eps_{i_1,\dots,i_q}(y) J_{i_1,\dots,i_q} \chi_1 \dots\chi_q
\]
for deterministic constants $\eps_{i_1,\dots,i_q}(y) \in \{-1,1\}$. Since the couplings $J_{i_1,\dots,i_q}$ are i.i.d. centered Gaussian, we find that $x$ and $yxy^{-1}$ have the same law. Since $x\mapsto M_x$ defines a representation of $\cC_n$, $M^x$ and $M^y M^x (M^y)^{-1}$ also have the same law. This completes the proof since $y\in\cG_{\SYK}$ was arbitrary.
\end{proof}

\begin{proof}[Proof of Lemma~\ref{lem:M-symmetrizing}]
We show that for any linear map $M:\cC_n\to\cC_n$, 
\begin{equation}
\label{eq:SYK-symmetrizing}
    \bbE^{g\in\cG_{\SYK}}[gMg^{-1}]=
    \bbE^{g\in\cG_{\SYK}^+}[gMg^{-1}]
    =
    \frac{\Tr(M)\cdot I_{2^n}}{2^n}.
\end{equation}
The first equality is clear because $gMg^{-1}=(-g)M(-g)^{-1}$ for all $g\in\cG_{\SYK}$, so we focus on the latter equality. Recall that $\cG_{\SYK}^+$ is a basis for $\cC_n$. In this basis, it is easy to see that each $M^g$ for $g\in \cG_{\SYK}^+$ acts by a signed permutation matrix which is the identity when $g=1$, and otherwise has all diagonal entries $0$ hence trace $0$. Since \eqref{eq:SYK-symmetrizing} is clear when $M$ is the identity, it suffices to establish for all non-identity elements $x\in \cG_{\SYK}^+$ the equality
\begin{equation}
\label{eq:SYK-symmetrizing-reduced}
    \sum_{g\in\cG_{\SYK}^+}gM^xg^{-1}
    =
    0.
\end{equation}
From the $S_n$ symmetry guaranteed by Proposition~\ref{prop:quantum-group}, we may without loss of generality assume that $x$ contains the term $\chi_n$, so that $x=\chi_{i_1} \dots\chi_{i_{q-1}}\chi_n$ for $1\leq i_1<\dots<i_{q-1}<n$. Next, let $\cG_{\SYK,-n}^+\subseteq \cG_{\SYK}^+$ consist of the monomials with no $\chi_n$ term. Since $\cG_{\SYK}^+=\cG_{\SYK,-n}^+\cup \lt(\cG_{\SYK,-n}^+\cdot \chi_n\rt)$, we find
\begin{align*}
    \sum_{g\in\cG_{\SYK}^+}
    gM^xg^{-1}
    &=
    \sum_{g\in\cG_{\SYK,-n}^+}
    gM^xg^{-1}+ g\chi_n M^x \chi_n^{-1} g^{-1} 
    \\
    &=
    \sum_{g\in\cG_{\SYK,-n}^+}
    g\big(M+\chi_n M^x \chi_n\big) g^{-1} 
    \\
    &=
    0.
\end{align*}
In the last step we used the identity $\chi_n M^x \chi_n=-M^x$ for $x=\chi_{i_1} \dots\chi_{i_{q-1}}\chi_n$. This holds because $q$ is even and $\chi_n\chi_{i_k}=-\chi_{i_k}\chi_n$ for all $1\leq k\leq q-1$. 
\end{proof}

\begin{remark}
\label{rem:quantum-general}
    In the spirit of \eqref{eq:Gaussian-H}, Corollary~\ref{cor:SYK} extends to more general formulations of the SYK model. Let $[Q]_2=[Q]\cap 2\bbN$ be the set of positive even integers at most $Q$ and consider
    \begin{equation}
    \label{eq:SYK-general}
        x=
        \sum_{q\in [Q]_2}
        i^{q/2}
        \sum_{1\leq i_1<\dots<i_q\leq n} J_{i_1,\dots,i_q} \chi_1 \dots\chi_q
    \end{equation}
    for independent centered Gaussians $J_{i_1,\dots,i_q}$ with arbitrary variances $\vec c=(c_{i_1,\dots,i_q})_{i_j\in [n],q\in [Q]_2}$. The analog of Corollary~\ref{eq:beta} generalizes to this setting, as indeed Corollary~\ref{cor:SYK}'s proof extends essentially verbatim.

    Further, Corollaries~\ref{cor:QSK} and  \ref{cor:SYK} admit a common generalization. This is not surprising if one observes that the algebra of linear endomorphisms $\End(\bbC^2,\bbC^2)$ is isomorphic to $\cC_2$. Indeed it follows from the relations \eqref{eq:pauli-relations} that an isomorphism is given by extending the map
    \begin{align*}
        I_2 &\mapsto 1,
        \\
        \sigma^x &\mapsto \chi_1,
        \\
        \sigma^y &\mapsto i\cdot \chi_1\chi_2,
        \\
        \sigma^z &\mapsto \chi_2
        .
    \end{align*}
    Therefore the quantum SK Hamiltonians considered previously live in $\cC_2^{\otimes m}$ while SYK Hamiltonians live in $\cC_n$. In general, for any positive integers $m$ and $(n_1,\dots,n_m)$, we may consider the algebra $\bigotimes_{j=1}^m \cC_{n_j}$ with $(\chi_i^j)_{i\in [n_j]}$ the generators of $\cC_{n_j}$.
    Then Theorem~\ref{thm:main-M} implies analogous free energy subadditivity for the random operators on $\bigotimes_{j=1}^m \cC_{n_j}$ given by
    \[
        M=
        \sum_{(q_1,\dots,q_m)\in [Q_1]_2\times\dots\times [Q_m]_2}
        i^{\sum_{j=1}^m q_j/2}
        \sum_{1\leq i^j_1<i^j_2<\dots<i^j_{q_j}\leq n_j}
        J_{\{(i^j_k):\,j\in [m],k\in [q_j]\}}
        \prod_{j=1}^{m}
        \big(\chi_{i^j_1}^j\dots\chi_{i^j_{q_j}}^j\big)
    \]
    where $J_{(\cdot)}\sim \cN(0,c_{(\cdot)})$ are independent centered Gaussians with arbitrary variances.
    The proof is identical to that of Corollary~\ref{cor:QSK}, with Lemma~\ref{lem:M-symmetrizing} used in the last step on each $\cC_{n_j}$ in place of \eqref{eq:gQg}.
\end{remark}

\section*{Acknowledgement}

We thank Sourav Chatterjee, Brice Huang and the anonymous referee for helpful comments. We were introduced to the SYK model by an excellent lecture of Ryan O'Donnell at the Simons Institute. This work was supported in part by NSF grant CCF2006489.

\appendix

\section{Suboptimal Alternate Approach for Mixed $p$-spin Models}
\label{sec:parisi}

Recall from Subsection~\ref{subsec:mixed-p-spin} that for a deterministic sequence $(c_2,\dots,c_P)\in\bbR_{\geq 0}^{P}$, the corresponding mixed $p$-spin Hamiltonian without external field\footnote{Corollary~\ref{cor:parisi} extends easily to handle an external field, but we omit this for convenience.} is given by
\[
    H_N(\bsig)
    =
    \sum_{p=2}^P
    \frac{1}{N^{(p-1)/2}}
    \sum_{1\leq i_1,\dots,i_p\leq N}
    J_{i_1,\dots,i_p}
    \sigma_{i_1}\dots\sigma_{i_p}
\]
for independent Gaussians $J_{i_1,\dots,i_p}\sim \cN(0,c_p)$ with variance $c_p$. The celebrated Parisi formula gives the limiting free energy in this model for the reference measure $\mu$ which is uniform on $\bsig\in\{-1,1\}^N$.

Surprisingly, we were unable to recover the free energy subadditivity in
\begin{equation}
    \label{eq:xi-appendix}
    \xi(x)=\sum_{p=2}^P c_p x^p
\end{equation}
directly from the Parisi formula. In this appendix we use it to show the weaker Corollary~\ref{cor:parisi} in which free energy is defined relative to \textbf{counting} measure on $\{-1,1\}^N$ instead of uniform measure. Counting measure is actually more commonly used to define free energy in Ising spin glasses since the corresponding partition function is just the sum $\sum_{\bsig\in \{-1,1\}^N} H_N(\bsig)$. (Indeed this discrepancy led to some confusion on our part, which motivated us to include this appendix.)
In Remark~\ref{rem:subadditive-fails} we discuss why our proof strategy seemingly cannot recover the stronger estimate of Theorem~\ref{thm:main}. It would be interesting to derive free energy subadditivity relative to uniform measure by using the Parisi formula in a different way.

Let $\cM_{[0,1]}$ denote the space of increasing and right-continuous functions $\zeta:[0,1]\to [0,1]$. To state the Parisi formula at inverse temperature $\beta=1$ (without loss of generality since $\beta$ can be absorbed into $\xi$), we define for $\zeta\in\cM_{[0,1]}$ the function $\Phi_{\xi,\zeta}:[0,1]\times\bbR\to \bbR_{\geq 0}$ as the solution to the non-linear PDE:
\begin{align}
    \label{eq:ParisiPDEdefn}
    \partial_t \Phi_{\xi,\zeta}(t,x)+\frac{1}{2}\xi''(t)\left(\partial_{xx}\Phi_{\xi,\zeta}(t,x)+\zeta(t)(\partial_x \Phi_{\xi,\zeta}(t,x))^2\right)&=0
    \\
\nonumber
    \Phi_{\xi,\zeta}(1,x)&=\log\cosh(x).
\end{align}
Existence and uniqueness of solutions are shown in \cite{auffinger2015parisi,jagannath2016dynamic} (in fact we will use the formula from \cite[Theorem 1]{auffinger2015parisi} below, with the term $\log 2$ omitted for the main statement with uniform measure). The Parisi functional for the Ising mixed $p$-spin model is defined by 
\begin{equation}
    \label{eq:def-parisi-functional-is}
    \Par_{\xi}(\zeta) 
    = 
    \Phi_{\xi,\zeta}(0,0) - \frac{1}{2}\int_{0}^1 t\xi''(t)\zeta(t) \diff{t}.
\end{equation}
Finally the Parisi formula \cite{parisi1979infinite,guerra2003broken,talagrand2006parisi,panchenko2013parisi,panchenko2014parisi,panchenko2013sherrington} states that
\[
    F(\xi)\equiv \lim_{N\to\infty} F_N(\xi)/N = \inf_{\zeta\in\cM_{[0,1]}} \Par_{\xi}(\zeta).
\]

We will use the more convenient Auffinger-Chen representation for the Parisi functional which we now describe. Given a filtration $\mathcal F=(\mathcal F_t)_{t\in [0,1]}$ and adapted standard Brownian $B(t)$, let $\mathcal D[0,1]$ denote the space of progressively measurable processes $(u_t)_{t\in [0,1]}$ such that $|u_t|\leq 1$ holds almost surely for each $t\in [0,1]$.
Define
\begin{align*}
    \cX_{\xi,\zeta}(u;B)&=\cY_{\xi,\zeta}(u;B)-\cZ_{\xi,\zeta}(u);
    \\
    \cY_{\xi,\zeta}(u;B)
    &\equiv 
    \Phi_{\xi,\zeta}\left(1,\int_{0}^{1} \zeta(t)\xi''(t)u(t)\de t
    +
    \int_{0}^{1} \sqrt{\xi''(t)}\de B(t)
    \right),
    \\
    \cZ_{\xi,\zeta}(u)&\equiv\frac{1}{2}\int_{0}^{1} \zeta(t)\xi''(t)u(t)^{2} \de t\,.
\end{align*}
Auffinger-Chen showed the Parisi functional can be represented as a stochastic control problem involving $\cX_{\xi,\zeta}$.

\begin{proposition}{\cite[Theorem 3]{auffinger2015parisi}}
\label{prop:optimal-control}
The function $\Phi_{\xi,\zeta}$ defined in \eqref{eq:ParisiPDEdefn} satisfies
\[
    \Phi_{\xi,\zeta}(0,0)= \max_{u\in\cD[0,1]} \E[\cX_{\xi,\zeta}(u;B)]\,.
\]  
\end{proposition}

We note that the statement in \cite{auffinger2015parisi} assumes $\mathcal F$ is the filtration generated by $B(t)$, but this is not necessary.

\begin{proposition}
\label{prop:parisi}
For mixed $p$-spin covariances $\xi_1,\xi_2$ as in \eqref{eq:xi-appendix} and $\zeta_1,\zeta_2\in \cM_{[0,1]}$, let
\begin{equation}
\label{eq:zeta-defn}
    \zeta(t)\equiv
    \frac{\zeta_1(t)\xi_1''(t)+\zeta_2(t)\xi_2''(t)}
    {\xi_1''(t)+\xi_2''(t)}\in \cM_{[0,1]}\,.
\end{equation}
Then we have the inequality
\[
    \Phi_{\xi_1+\xi_2,\zeta}(0,0)
    \leq
    \Phi_{\xi_1,\zeta_1}(0,0)
    +
    \Phi_{\xi_2,\zeta_2}(0,0)+\log(2)
    \,.
\]
\end{proposition}

\begin{proof}
Let $B_1(t)$ and $B_2(t)$ be independent standard Brownian motions generating together the filtration $\mathcal F=(\mathcal F_t)_{t\geq 0}$. Define $\zeta$ as in \eqref{eq:zeta-defn} and set
\begin{align*}
    B(t)&=\sqrt{\frac{\xi_1(t)}{\xi_1(t)+\xi_2(t)}}B_1(t)+\sqrt{\frac{\xi_2(t)}{\xi_1(t)+\xi_2(t)}}B_2(t).
\end{align*}
Note that $B(t)$ is also a standard Brownian motion. We choose $u\in\cD[0,1]$ to maximize $\cX_{\xi_1+\xi_2}(u,0)$ with driving Brownian motion $B_t$, so that $\Phi_{\xi_1+\xi_2,\zeta}(0,0)=\cX_{\xi_1+\xi_2,\zeta}(u;B)$. Proposition~\ref{prop:optimal-control} implies that for $i\in \{1,2\}$, we have
\begin{equation}
\label{eq:ubi}
    \E[\cX_{\xi_i,\zeta_i}(u;B_i)]
    \leq 
    \Phi_{\xi_i,\zeta_i}(0,0)\,.
\end{equation}
Therefore it suffices to prove that almost surely,
\begin{equation}
\label{eq:Parisi-PDE-inequality}
    \cX_{\xi_1+\xi_2,\zeta}(u;B)
    \leq
    \cX_{\xi_1,\zeta_1}(u;B_1)
    +
    \cX_{\xi_2,\zeta_2}(u;B_2)+\log(2).
\end{equation}
We first have
\[
    \cZ_{\xi_1+\xi_2,\zeta}(u)
    =
    \cZ_{\xi_1,\zeta_1}(u)
    +
    \cZ_{\xi_2,\zeta_2}(u)
\]
by definition. Next, the inequality
\[
    \cY_{\xi_1+\xi_2,\zeta}(u;B)
    \leq
    \cY_{\xi_1,\zeta_1}(u;B_1)
    +
    \cY_{\xi_2,\zeta_2}(u;B_2)+\log(2)
\]
follows from the identity
\begin{align*} 
    \int_{0}^{1} \zeta(t)\big(\xi_1''(t)+\xi_2''(t)\big)u(t)\de t
    +
    \int_{0}^{1} \sqrt{\xi_1''(t)+\xi_2''(t)}\de B(t)
    &=
    \int_{0}^{1} \zeta_1(t)\xi_1''(t)u(t)\de t
    +
    \int_{0}^{1} \sqrt{\xi_1''(t)}\de B_1(t)
    \\
    &\quad
    +
    \int_{0}^{1} \zeta_2(t)\xi_2''(t)u(t)\de t
    +
    \int_{0}^{1} \sqrt{\xi_2''(t)}\de B_2(t)
\end{align*}
and the easily verified subadditivity of the function $\log\big(2\cosh(x)\big)$. This concludes the proof.
\end{proof}

\begin{corollary}
\label{cor:parisi}
    The free energy in the mixed $p$-spin model satisfies $F(\xi_1+\xi_2)\leq F(\xi_1)+F(\xi_2)+\log(2)$.
\end{corollary}

\begin{proof}
    Choose $\zeta_1,\zeta_2\in\cM_{[0,1]}$ to minimize the respective Parisi functionals, i.e.
    \[
        F(\xi_i)=\Par_{\xi_i}(\zeta_i),\quad i\in \{1,2\}.
    \]
    Then with $\zeta$ as in \eqref{eq:zeta-defn}. we have
    \begin{align*}
        F(\xi_1+\xi_2)&\leq\Par_{\xi_1+\xi_2}(\zeta)
        \\
        &=
        \Phi_{\xi_1+\xi_2,\zeta}(0,0)-\frac{1}{2}\int_0^1 t\big(\xi_1''(t)+\xi_2''(t)\big)\zeta(t)\de t
        \\
        &\leq
        \Phi_{\xi_1,\zeta_1}(0,0)+\Phi_{\xi_2,\zeta_2}(0,0)
        +\log(2)
        -
        \frac{1}{2}\int_0^1 t\big(\xi_1''(t)\zeta_1(t)+\xi_2''(t)\zeta_2(t)\big)\de t
        \\
        &=
        \Par_{\xi_1}(\zeta_1)
        +
        \Par_{\xi_2}(\zeta_2) +\log(2)
        \\
        &=
        F(\xi_1)+F(\xi_2) +\log(2).
    \end{align*}
\end{proof}

\begin{remark}
\label{rem:subadditive-fails}

There is a good reason that our argument above cannot recover the stronger estimate of Theorem~\ref{thm:main}. Since the process $u$ used in proving Proposition~\ref{prop:parisi} is adapted to the filtration generated by $B(t)$, it cannot agree with any nontrivial process adapted to the filtrations generated by $B_1(t)$ or $B_2(t)$. This means that \eqref{eq:ubi} and hence Proposition~\ref{prop:parisi} essentially never hold with equality. By contrast the bound in Corollary~\ref{cor:temperature} does hold with equality at high temperature.

Let us also mention that our proof of Proposition~\ref{prop:parisi} has a similar spirit to that of \cite[Theorem 4]{auffinger2015parisi}, which used Proposition~\ref{prop:optimal-control} to establish the strict convexity of the Parisi functional in $\zeta$. In fact the obstruction just outlined resembles their proof that the convexity is strict.
\end{remark}

\bibliographystyle{plain}
\bibliography{all-bib}

\begin{thebibliography}{10}

\bibitem{achlioptas2004threshold}
Dimitris Achlioptas and Yuval Peres.
\newblock {The Threshold for Random $k$-SAT Is $2^k \log(2)-O(k)$}.
\newblock {\em Journal of the American Mathematical Society}, pages 947--973,
  2004.

\bibitem{adhikari2020free}
Arka Adhikari and Christian Brennecke.
\newblock {Free energy of the quantum Sherrington--Kirkpatrick spin-glass model
  with transverse field}.
\newblock {\em Journal of Mathematical Physics}, 61(8):083302, 2020.

\bibitem{aizenman1990rounding}
Michael Aizenman and Jan Wehr.
\newblock Rounding effects of quenched randomness on first-order phase
  transitions.
\newblock {\em Communications in Mathematical Physics}, 130(3):489--528, 1990.

\bibitem{arous2019landscape}
Gerard~Ben Arous, Song Mei, Andrea Montanari, and Mihai Nica.
\newblock The landscape of the spiked tensor model.
\newblock {\em Communications on Pure and Applied Mathematics},
  72(11):2282--2330, 2019.

\bibitem{auffinger2015parisi}
Antonio Auffinger and Wei-Kuo Chen.
\newblock {The Parisi formula has a unique minimizer}.
\newblock {\em Communications in Mathematical Physics}, 335(3):1429--1444,
  2015.

\bibitem{baik2005phase}
Jinho Baik, G{\'e}rard Ben~Arous, and Sandrine P{\'e}ch{\'e}.
\newblock Phase transition of the largest eigenvalue for nonnull complex sample
  covariance matrices.
\newblock {\em Annals of Probability}, pages 1643--1697, 2005.

\bibitem{barbier2022marginals}
Jean Barbier and Manuel S{\'a}enz.
\newblock Marginals of a spherical spin glass model with correlated disorder.
\newblock {\em Electronic Communications in Probability}, 27:1--12, 2022.

\bibitem{barra2015multi}
Adriano Barra, Pierluigi Contucci, Emanuele Mingione, and Daniele Tantari.
\newblock {Multi-species mean field spin glasses. Rigorous results}.
\newblock In {\em Annales Henri Poincar{\'e}}, volume~16, pages 691--708.
  Springer, 2015.

\bibitem{bates2021crisanti}
Erik Bates and Youngtak Sohn.
\newblock Crisanti--sommers formula and simultaneous symmetry breaking in
  multi-species spherical spin glasses.
\newblock {\em Communications in Mathematical Physics}, 394(3):1101--1152,
  2022.

\bibitem{bates2022free}
Erik Bates and Youngtak Sohn.
\newblock Free energy in multi-species mixed p-spin spherical models.
\newblock {\em Electronic Journal of Probability}, 27:1--75, 2022.

\bibitem{bayati2010combinatorial}
Mohsen Bayati, David Gamarnik, and Prasad Tetali.
\newblock Combinatorial approach to the interpolation method and scaling limits
  in sparse random graphs.
\newblock {\em The Annals of Probability}, 41(6):4080--4115, 2013.

\bibitem{benaych2011eigenvalues}
Florent Benaych-Georges and Raj~Rao Nadakuditi.
\newblock The eigenvalues and eigenvectors of finite, low rank perturbations of
  large random matrices.
\newblock {\em Advances in Mathematics}, 227(1):494--521, 2011.

\bibitem{bhattacharya2016high}
Bhaswar~B Bhattacharya and Subhabrata Sen.
\newblock High temperature asymptotics of orthogonal mean-field spin glasses.
\newblock {\em Journal of Statistical Physics}, 162(1):63--80, 2016.

\bibitem{bramson1978maximal}
Maury~D Bramson.
\newblock {Maximal displacement of branching Brownian motion}.
\newblock {\em Communications on Pure and Applied Mathematics}, 31(5):531--581,
  1978.

\bibitem{bricmont1988phase}
Jean Bricmont and Antti Kupiainen.
\newblock {Phase transition in the 3d random field Ising model}.
\newblock {\em Communications in Mathematical Physics}, 116(4):539--572, 1988.

\bibitem{brunet2011branching}
{\'E}ric Brunet and Bernard Derrida.
\newblock A branching random walk seen from the tip.
\newblock {\em Journal of Statistical Physics}, 143(3):420--446, 2011.

\bibitem{buffet1993directed}
Emmanuel Buffet, A~Patrick, and Joe~V Pul{\'e}.
\newblock Directed polymers on trees: a martingale approach.
\newblock {\em Journal of Physics A: Mathematical and General}, 26(8):1823,
  1993.

\bibitem{chatterjee2019central}
Sourav Chatterjee.
\newblock {Central limit theorem for the free energy of the random field Ising
  model}.
\newblock {\em Journal of Statistical Physics}, 175(1):185--202, 2019.

\bibitem{chauvin1997boltzmann}
B~Chauvin and A~Rouault.
\newblock Boltzmann-{G}ibbs weights in the branching random walk.
\newblock In {\em Classical and Modern Branching Processes}, pages 41--50.
  Springer, 1997.

\bibitem{chen2014chaos}
Wei-Kuo Chen.
\newblock Chaos in the mixed even-spin models.
\newblock {\em Communications in Mathematical Physics}, 328(3):867--901, 2014.

\bibitem{chen2017variational}
Wei-Kuo Chen.
\newblock {Variational representations for the Parisi functional and the
  two-dimensional Guerra--Talagrand bound}.
\newblock {\em The Annals of Probability}, 45(6A):3929--3966, 2017.

\bibitem{chen2019phase}
Wei-Kuo Chen.
\newblock {Phase transition in the spiked random tensor with Rademacher prior}.
\newblock {\em The Annals of Statistics}, 47(5):2734--2756, 2019.

\bibitem{chen2019suboptimality}
Wei-Kuo Chen, David Gamarnik, Dmitry Panchenko, Mustazee Rahman, et~al.
\newblock Suboptimality of local algorithms for a class of max-cut problems.
\newblock {\em The Annals of Probability}, 47(3):1587--1618, 2019.

\bibitem{chen2018energy}
Wei-Kuo Chen, Madeline Handschy, and Gilad Lerman.
\newblock On the energy landscape of the mixed even p-spin model.
\newblock {\em Probability Theory and Related Fields}, 171(1-2):53--95, 2018.

\bibitem{chen2015disorder}
Wei-Kuo Chen, Hsi-Wei Hsieh, Chii-Ruey Hwang, and Yuan-Chung Sheu.
\newblock Disorder chaos in the spherical mean-field model.
\newblock {\em Journal of Statistical Physics}, 160(2):417--429, 2015.

\bibitem{chen2018disorder}
Wei-Kuo Chen and Dmitry Panchenko.
\newblock Disorder chaos in some diluted spin glass models.
\newblock {\em The Annals of Applied Probability}, 28(3):1356--1378, 2018.

\bibitem{chen2017parisi}
Wei-Kuo Chen and Arnab Sen.
\newblock Parisi formula, disorder chaos and fluctuation for the ground state
  energy in the spherical mixed p-spin models.
\newblock {\em Communications in Mathematical Physics}, 350(1):129--173, 2017.

\bibitem{coja2016asymptotic}
Amin Coja-Oghlan and Konstantinos Panagiotou.
\newblock {The asymptotic $k$-SAT threshold}.
\newblock {\em Advances in Mathematics}, 288:985--1068, 2016.

\bibitem{contucci2007correlation}
Pierluigi Contucci and Joel Lebowitz.
\newblock Correlation inequalities for spin glasses.
\newblock In {\em Annales Henri Poincare}, volume~8, pages 1461--1467.
  Springer, 2007.

\bibitem{crawford2007thermodynamics}
Nicholas Crawford.
\newblock Thermodynamics and universality for mean field quantum spin glasses.
\newblock {\em Communications in Mathematical Physics}, 274(3):821--839, 2007.

\bibitem{derrida1985generalization}
Bernard Derrida.
\newblock A generalization of the random energy model which includes
  correlations between energies.
\newblock {\em Journal de Physique Lettres}, 46(9):401--407, 1985.

\bibitem{derrida1986solution}
Bernard Derrida and E~Gardner.
\newblock Solution of the generalised random energy model.
\newblock {\em Journal of Physics C: Solid State Physics}, 19(13):2253, 1986.

\bibitem{diestel2014joys}
Joe Diestel and Angela Spalsbury.
\newblock {\em {The Joys of Haar Measure}}.
\newblock American Mathematical Soc., 2014.

\bibitem{ding2016satisfiability}
Jian Ding, Allan Sly, and Nike Sun.
\newblock {Satisfiability threshold for random regular NAE-SAT}.
\newblock {\em Communications in Mathematical Physics}, 341(2):435--489, 2016.

\bibitem{ding2022proof}
Jian Ding, Allan Sly, and Nike Sun.
\newblock Proof of the satisfiability conjecture for large $k$.
\newblock {\em Annals of Mathematics}, 196(1):1--388, 2022.

\bibitem{ding2019capacity}
Jian Ding and Nike Sun.
\newblock Capacity lower bound for the {I}sing perceptron.
\newblock In {\em Proceedings of the 51st Annual ACM SIGACT Symposium on Theory
  of Computing}, pages 816--827, 2019.

\bibitem{ding2021exponential}
Jian Ding and Jiaming Xia.
\newblock {Exponential decay of correlations in the two-dimensional random
  field Ising model}.
\newblock {\em Inventiones mathematicae}, 224(3):999--1045, 2021.

\bibitem{donoho2018optimal}
David~L Donoho, Matan Gavish, and Iain~M Johnstone.
\newblock Optimal shrinkage of eigenvalues in the spiked covariance model.
\newblock {\em Annals of statistics}, 46(4):1742, 2018.

\bibitem{edwards1975theory}
Samuel~Frederick Edwards and Phil~W Anderson.
\newblock Theory of spin glasses.
\newblock {\em Journal of Physics F: Metal Physics}, 5(5):965, 1975.

\bibitem{fan2022tap}
Zhou Fan, Yufan Li, and Subhabrata Sen.
\newblock {TAP} equations for orthogonally invariant spin glasses at high
  temperature.
\newblock {\em arXiv preprint arXiv:2202.09325}, 2022.

\bibitem{fan2021replica}
Zhou Fan and Yihong Wu.
\newblock The replica-symmetric free energy for {I}sing spin glasses with
  orthogonally invariant couplings.
\newblock {\em arXiv preprint arXiv:2105.02797}, 2021.

\bibitem{forrester2014golden}
Peter~J Forrester and Colin~J Thompson.
\newblock {The Golden-Thompson Inequality: Historical Aspects and Random Matrix
  Applications}.
\newblock {\em Journal of Mathematical Physics}, 55(2):023503, 2014.

\bibitem{frohlich1987some}
J{\"u}rg Fr{\"o}hlich and Bogus{\l}aw Zegarlinski.
\newblock {Some comments on the Sherrington-Kirkpatrick model of spin glasses}.
\newblock {\em Communications in mathematical physics}, 112(4):553--566, 1987.

\bibitem{GJW20}
David Gamarnik, Aukosh Jagannath, and Alexander~S Wein.
\newblock Low-degree hardness of random optimization problems.
\newblock In {\em 2020 IEEE 61st Annual Symposium on Foundations of Computer
  Science (FOCS)}, pages 131--140. IEEE, 2020.

\bibitem{golden1965lower}
Sidney Golden.
\newblock Lower bounds for the helmholtz function.
\newblock {\em Physical Review}, 137(4B):B1127, 1965.

\bibitem{goldschmidt1990ising}
Yadin~Y Goldschmidt and Pik-Yin Lai.
\newblock Ising spin glass in a transverse field: {R}eplica-symmetry-breaking
  solution.
\newblock {\em Physical review letters}, 64(21):2467, 1990.

\bibitem{guerra2003broken}
Francesco Guerra.
\newblock Broken replica symmetry bounds in the mean field spin glass model.
\newblock {\em Communications in Mathematical Physics}, 233(1):1--12, 2003.

\bibitem{guerra2002thermodynamic}
Francesco Guerra and Fabio~Lucio Toninelli.
\newblock The thermodynamic limit in mean field spin glass models.
\newblock {\em Communications in Mathematical Physics}, 230(1):71--79, 2002.

\bibitem{huang2018convergence}
Brice Huang.
\newblock Convergence of maximum bisection ratio of sparse random graphs.
\newblock {\em Electronic Communications in Probability}, 23:1--10, 2018.

\bibitem{huang2021tight}
Brice Huang and Mark Sellke.
\newblock {Tight Lipschitz Hardness for Optimizing Mean Field Spin Glasses}.
\newblock {\em arXiv preprint arXiv:2110.07847}, 2021.

\bibitem{imry1975random}
Yoseph Imry and Shang-keng Ma.
\newblock Random-field instability of the ordered state of continuous symmetry.
\newblock {\em Physical Review Letters}, 35(21):1399, 1975.

\bibitem{jagannath2016branching}
Aukosh Jagannath.
\newblock On the overlap distribution of branching random walks.
\newblock {\em Electronic Journal of Probability}, 21:1--16, 2016.

\bibitem{jagannath2022existence}
Aukosh Jagannath and Patrick Lopatto.
\newblock Existence of the free energy for heavy-tailed spin glasses.
\newblock {\em arXiv preprint arXiv:2211.09879}, 2022.

\bibitem{jagannath2016dynamic}
Aukosh Jagannath and Ian Tobasco.
\newblock {A Dynamic Programming Approach to the Parisi Functional}.
\newblock {\em Proceedings of the American Mathematical Society},
  144(7):3135--3150, 2016.

\bibitem{johnstone2001distribution}
Iain~M Johnstone.
\newblock On the distribution of the largest eigenvalue in principal components
  analysis.
\newblock {\em The Annals of Statistics}, 29(2):295--327, 2001.

\bibitem{SYK-talk}
Alexei Kitaev.
\newblock {``A Simple Model Of Quantum Holography", talks at KITP, April 7,
  2015 and May 27}, 2015.
\newblock
  \href{https://online.kitp.ucsb.edu/online/entangled15/kitaev/}{https://online.kitp.ucsb.edu/online/entangled15/kitaev/},
  \href{https://online.kitp.ucsb.edu/online/entangled15/kitaev2/}{https://online.kitp.ucsb.edu/online/entangled15/kitaev2/}.

\bibitem{kivimae2021ground}
Pax Kivimae.
\newblock The ground state energy and concentration of complexity in spherical
  bipartite models.
\newblock {\em arXiv preprint arXiv:2107.13138}, 2021.

\bibitem{ko2020free}
Justin Ko.
\newblock Free energy of multiple systems of spherical spin glasses with
  constrained overlaps.
\newblock {\em Electronic Journal of Probability}, 25:1--34, 2020.

\bibitem{lalley1987conditional}
Steven~P Lalley and Thomas Sellke.
\newblock {A conditional limit theorem for the frontier of a branching Brownian
  motion}.
\newblock {\em The Annals of Probability}, pages 1052--1061, 1987.

\bibitem{leschke2021free}
Hajo Leschke, Sebastian Rothlauf, Rainer Ruder, and Wolfgang Spitzer.
\newblock {The free energy of a quantum Sherrington--Kirkpatrick spin-glass
  model for weak disorder}.
\newblock {\em Journal of Statistical Physics}, 182(3):1--41, 2021.

\bibitem{marinari1994replica}
Enzo Marinari, Giorgio Parisi, and Felix Ritort.
\newblock {Replica field theory for deterministic models: II. A non-random spin
  glass with glassy behaviour}.
\newblock {\em Journal of Physics A: Mathematical and General}, 27(23):7647,
  1994.

\bibitem{mourrat2021nonconvex}
Jean-Christophe Mourrat.
\newblock Nonconvex interactions in mean-field spin glasses.
\newblock {\em Probability and Mathematical Physics}, 2(2):281--339, 2021.

\bibitem{nam2022one}
Danny Nam, Allan Sly, and Youngtak Sohn.
\newblock {One-step replica symmetry breaking of random regular NAE-SAT}.
\newblock In {\em 2021 IEEE 62nd Annual Symposium on Foundations of Computer
  Science (FOCS)}, pages 310--318. IEEE, 2022.

\bibitem{newman1997metastate}
Charles~M Newman and Daniel~L Stein.
\newblock Metastate approach to thermodynamic chaos.
\newblock {\em Physical Review E}, 55(5):5194, 1997.

\bibitem{panchenko2013parisi}
Dmitry Panchenko.
\newblock {The Parisi ultrametricity conjecture}.
\newblock {\em Annals of Mathematics}, pages 383--393, 2013.

\bibitem{panchenko2013sherrington}
Dmitry Panchenko.
\newblock {\em {The Sherrington-Kirkpatrick model}}.
\newblock Springer Science \& Business Media, 2013.

\bibitem{panchenko2014parisi}
Dmitry Panchenko.
\newblock {The Parisi formula for mixed $ p $-spin models}.
\newblock {\em The Annals of Probability}, 42(3):946--958, 2014.

\bibitem{panchenko2015free}
Dmitry Panchenko.
\newblock {The free energy in a multi-species Sherrington--Kirkpatrick model}.
\newblock {\em The Annals of Probability}, 43(6):3494--3513, 2015.

\bibitem{panchenko2018free}
Dmitry Panchenko.
\newblock Free energy in the mixed $ p $-spin models with vector spins.
\newblock {\em The Annals of Probability}, 46(2):865--896, 2018.

\bibitem{panchenko2007overlap}
Dmitry Panchenko and Michel Talagrand.
\newblock On the overlap in the multiple spherical sk models.
\newblock {\em The Annals of Probability}, 35(6):2321--2355, 2007.

\bibitem{parisi1979infinite}
Giorgio Parisi.
\newblock Infinite number of order parameters for spin-glasses.
\newblock {\em Physical Review Letters}, 43(23):1754, 1979.

\bibitem{parisi1995mean}
Giorgio Parisi and Marc Potters.
\newblock Mean-field equations for spin models with orthogonal interaction
  matrices.
\newblock {\em Journal of Physics A: Mathematical and General}, 28(18):5267,
  1995.

\bibitem{perry2016statistical}
Amelia Perry, Alexander~S. Wein, and Afonso~S. Bandeira.
\newblock {Statistical limits of spiked tensor models}.
\newblock {\em Annales de l'Institut Henri Poincaré, Probabilités et
  Statistiques}, 56(1):230 -- 264, 2020.

\bibitem{perry2018optimality}
Amelia Perry, Alexander~S Wein, Afonso~S Bandeira, and Ankur Moitra.
\newblock {Optimality and sub-optimality of PCA I: Spiked random matrix
  models}.
\newblock {\em The Annals of Statistics}, 46(5):2416--2451, 2018.

\bibitem{ray1989sherrington}
Pulak Ray, Bikas~K Chakrabarti, and Arunava Chakrabarti.
\newblock {Sherrington-Kirkpatrick model in a transverse field: Absence of
  replica symmetry breaking due to quantum fluctuations}.
\newblock {\em Physical Review B}, 39(16):11828, 1989.

\bibitem{rosenhaus2019introduction}
Vladimir Rosenhaus.
\newblock An introduction to the {SYK} model.
\newblock {\em Journal of Physics A: Mathematical and Theoretical},
  52(32):323001, 2019.

\bibitem{sachdev1993gapless}
Subir Sachdev and Jinwu Ye.
\newblock Gapless spin-fluid ground state in a random quantum heisenberg
  magnet.
\newblock {\em Physical review letters}, 70(21):3339, 1993.

\bibitem{subag2021second}
Eliran Subag.
\newblock On the second moment method and {RS} phase of multi-species spherical
  spin glasses.
\newblock {\em arXiv preprint arXiv:2111.07133}, 2021.

\bibitem{subag2021tap1}
Eliran Subag.
\newblock {TAP approach for multi-species spherical spin glasses I: general
  theory}.
\newblock {\em arXiv preprint arXiv:2111.07132}, 2021.

\bibitem{subag2021tap}
Eliran Subag.
\newblock {TAP approach for multi-species spherical spin glasses II: the free
  energy of the pure models}.
\newblock {\em arXiv preprint arXiv:2111.07134}, 2021.

\bibitem{talagrand2006parisi}
Michel Talagrand.
\newblock {The Parisi formula}.
\newblock {\em Annals of Mathematics}, pages 221--263, 2006.

\bibitem{talagrand2011mean1}
Michel Talagrand.
\newblock {\em Mean Field Models for Spin Glasses. Volume I: Basic Examples},
  volume~54.
\newblock Springer Science \& Business Media, 2011.

\bibitem{talagrand2011mean2}
Michel Talagrand.
\newblock {\em Mean Field Models for Spin Glasses. Volume II: Advanced
  Replica-Symmetry and Low Temperature}, volume~55.
\newblock Springer Science \& Business Media, 2011.

\bibitem{thompson1965inequality}
Colin~J Thompson.
\newblock Inequality with applications in statistical mechanics.
\newblock {\em Journal of Mathematical Physics}, 6(11):1812--1813, 1965.

\end{thebibliography}

\end{document}